\documentclass[a4paper, svgnames, 10pt]{amsart}

\usepackage[a-1b]{pdfx}

\usepackage{
    amssymb,
    mathtools,
    xcolor,
    booktabs,
    graphicx,
    upgreek,
    comment,
    geometry,
}

\usepackage[ruled, noline]{algorithm2e}
\SetProcNameSty{textsc}\SetFuncSty{textsc}
\SetKwInOut{Input}{input}\SetKwInOut{Output}{output}
\SetKw{Return}{return}
\SetKw{KwTo}{to}
\DontPrintSemicolon

\usepackage[english]{babel}

\usepackage[final]{fixme}

\usepackage{imakeidx}
\makeindex[title=Index of Symbols and their Types] 

\usepackage{hyperref}
\hypersetup{colorlinks, allcolors=MediumBlue}
\newcommand{\msc}[1]{\href{https://zbmath.org/classification/?q=#1}{#1}}
\newcommand{\doi}[1]{doi:\href{https://doi.org/#1}{#1}}

\usepackage[noabbrev, capitalise]{cleveref}
\crefdefaultlabelformat{#2\textup{#1}#3}

\crefname{main} {Theorem}       {Theorems}
\crefname{thm}  {Theorem}       {Theorems}
\crefname{lem}  {Lemma}         {Lemmas}
\crefname{prop} {Proposition}   {Propositions}
\crefname{fig}  {Figure}        {Figures}
\crefname{tbl}  {Table}         {Tables}
\crefname{rmk}  {Remark}        {Remarks}
\crefname{que}  {Question}      {Questions}

\theoremstyle{plain}
\newtheorem{main} {Theorem}

\newtheorem{thm} {Theorem} [section]
\newtheorem{prop}   [thm] {Proposition}
\newtheorem{cor}    [thm] {Corollary}
\newtheorem{lem}    [thm] {Lemma}
\theoremstyle{definition}
\newtheorem{rmk}    [thm] {Remark}
\newtheorem{que}    [thm] {Question}

\numberwithin{equation}{section}

\newcommand{\FF}{\mathbb{F}}

\DeclarePairedDelimiterX\set[1]\lbrace\rbrace{\,\def\given{\,\delimsize\vert\,}#1\,}
\DeclarePairedDelimiterX\gen[1]\langle\rangle{\,\def\given{\,\delimsize\vert\,}#1\,}
\DeclareMathOperator{\Hom}{Hom}

\newcommand{\SmallGroup}{\texttt{SmallGroup}}
\newcommand{\Cent}{C}               
\newcommand{\Jen}{D}                
\newcommand{\rad}[1]{\sqrt{#1}}     
\newcommand{\Dih}[1]{D_{#1}}        
\newcommand{\Quat}[1]{Q_{#1}}       
\newcommand{\Z}[1]{Z(#1)}           
\newcommand{\Frat}[1]{\Phi(#1)}     
\newcommand{\rank}[1]{d(#1)}        
\newcommand{\rH}{\mathrm{H}}        
\newcommand{\HH}{\mathrm{HH}}       
\newcommand{\cA}{\mathcal{A}}
\newcommand{\cB}{\mathcal{B}}

\begin{document}

\title[Geometric approach to the modular isomorphism problem]{Geometric approach to\\ the modular isomorphism problem:\\ groups of order $64$}

\author[L.~Margolis]{Leo Margolis\,
}
\address[Leo Margolis]
{Universidad Aut\'onoma de Madrid, Departamento de Matem\'aticas,  C/ Francisco Tom\'as y Valiente 7, Facultad de Ciencias, m\'odulo 17, 28049 Madrid, Spain.}
\email{leo.margolis@uam.es}

\author[T.~Sakurai]{Taro Sakurai\,
}
\address[Taro Sakurai]
{Department of Mathematics and Informatics, Graduate School of Science, Chiba University, 1-33, Yayoi-cho, Inage-ku, Chiba-shi, Chiba, 263-8522, Japan.}
\email{tsakurai@math.s.chiba-u.ac.jp}

\thanks{
    The first author acknowledges financial support from the Spanish Ministry of Science and Innovation, through the Ram\'on y Cajal grant program.
    We also thank Tommy Hofmann and Max Horn for their help in our first use of \texttt{Oscar} and Diego Garc\'ia-Lucas for pointing out an error in an earlier version of this paper.
}



\date{\today}

\begin{abstract}
    We introduce a procedure based on computational algebraic geometry to determine whether two algebras are isomorphic.
    We then apply it to show that if $R$ is a commutative unital ring in which $2$ is not invertible, $G$ is a group of order dividing $64$ and $H$ some group, then an isomorphism of unital algebras $RG \cong RH$ implies an isomorphism of groups $G \cong H$.
\end{abstract}

\maketitle

\section{Introduction}

\index{$R$ : Commutative Ring}
\index{$G$ : Finite Group}
\index{$H$ : Finite Group}
The Modular Isomorphism Problem asks whether an isomorphism of two group algebras of finite $p$-groups over a field of characteristic $p > 0$ implies an isomorphism of the groups.
\index{$p$ : Prime Number}
Although some negative solutions have been found recently \cite{GarciaLucasMargolisDelRio22, MargolisSakurai25, BaginskiZabielski25}, these form a rather isolated class;
the problem is far from being well understood in general and many positive solutions are still known only over prime fields \cite{MargolisSakuraiSaoPaulo}.
This paper presents a new algorithmic idea to attack the problem based on computational algebraic geometry and results from its implementation.

To explain the idea, let $\Lambda$ and $\Gamma$ be free $R$-algebras of equal finite rank over a nonzero commutative unital ring $R$.
\index{$\Lambda$ : Free Algebra}
\index{$\Gamma$ : Free Algebra}
If we want to decide, whether $\Lambda$ and $\Gamma$ are isomorphic, an idea could be to try to see, if a subset of $R$-linearly independent elements in $\Gamma$ satisfies the relations between the elements of a basis $\cA$ of $\Lambda$.
\index{$A$ @ $\cA$ : Basis}
Consider an isomorphism of algebras from $\Lambda$ to $\Gamma$ and represent the images of elements in $\cA$ as generic elements with respect to a basis $\cB$ of $\Gamma$.
Then we obtain a matrix $[\upalpha_{a b}]$ in a polynomial algebra over $R$ with $n^2$ variables, where $n$ is the $R$-rank of $\Lambda$ and $\Gamma$.
\index{$B$ @ $\cB$ : Basis}
\index{$n$ : Positive Integer}
The algebraic relations of the elements can then be formulated as polynomials in these variables.
Under an evaluation map $R[\upalpha_{a b}] \to R$ defined by $\upalpha_{a b} \mapsto \alpha_{a b}$, the concrete matrix $[\alpha_{a b}]$ is invertible and so is the determinant $|\alpha_{a b}|$.
Hence $\Lambda$ and $\Gamma$ are not isomorphic if and only if the determinant $|\upalpha_{a b}|$, as an element of the polynomial algebra modulo the kernel, is not invertible.
This will be particularly the case, when a power of the determinant $|\upalpha_{a b}|$ lies in the ideal $I$ generated by the algebraic relations, i.e. when it is an element in the radical of $I$.
\index{$I$ : Ideal}
The adjective \emph{geometric} in the title basically refers to this application of the radical of ideals in polynomial algebras.

Although this approach seems rather impractical at first glance, we show that in the context of the Modular Isomorphism Problem it can be sufficiently simplified to obtain:

\begin{main}\label{main:64}
    Let $R$ be a commutative unital ring in which $2$ is not invertible and let $G$ and $H$ be groups.
    If the order of $G$ divides $64$, then an isomorphism of group algebras $RG \cong RH$ as unital algebras implies $G \cong H$ as groups.
\end{main}

This is not the first algorithm practically implemented to study the Modular Isomorphism Problem.
In fact, it is the third after \cite{Wursthorn93} and \cite{Eick08}.
The advantage in our approach is that it can detect the non-isomorphism of modular group algebras over all fields of characteristic $p$ at once by finite computations, while the earlier algorithms were only confirming this over a fixed finite field, usually the prime field.
Admittedly, our algorithm cannot detect whether two algebras are isomorphic over a given field, but it can determine whether an isomorphism exists over some finite field extension.
We remark that groups of order $64$ were also studied in previous approaches, but these took place only over the prime field, cf.~\cref{sec:ComputationsOrder64} for details.

After presenting the theory behind our idea in more detail in \cref{sec:Approach}, we illustrate it in a small example in \cref{sec:Order8}.
The computations for groups of order $64$ are presented in \cref{sec:ComputationsOrder64}.
\cref{sec:Appendix} then contains the data we need to know about the groups of order $64$ to reduce our previous calculations to only the $18$ cases studied in \cref{sec:ComputationsOrder64}.

\section{Geometric approach}\label{sec:Approach}
We will explain the basics of our geometric ideas to study the Modular Isomorphism Problem (MIP) in a more general setting and present a basic framework of the computation.
In fact, the idea could be presented in an even more general setting, but we choose a situation close enough to our intended application to the MIP.

Let $\Lambda$ be a finite-dimensional nilpotent algebra over a field $F$, where $n \geq 0$ is minimal such that $\Lambda^{n+1} = 0$.
\index{$F$ : Field}
Let $\cA$ be a basis of $\Lambda$ as an $F$-vector space.
Recall that the structure constants $(\sigma_{xy}^z)$ of $\Lambda$ with respect to $\cA$ are defined by $xy = \sum_{z \in \cA} \sigma_{xy}^z z$ for all $x, y \in \cA$.
We call $\cA$ a filtered basis if there exist subsets $\cA_1,\dotsc,\cA_n$ of $\cA$ such that the natural images of $\cA_1$ form a basis of $\Lambda/\Lambda^2$ as an $F$-vector space, the images of $\cA_2$ form a basis of $\Lambda^2/\Lambda^3$ etc. Note that every finite-dimensional nilpotent algebra has a filtered basis.
If $\cA$ is a filtered basis of $\Lambda$, we write $\cA_k$ for a subset of $\cA$ providing a basis of $\Lambda^k/\Lambda^{k+1}$.

To test whether two finite-dimensional nilpotent algebras are isomorphic, we simply use determinants---or more generally, maximal minors---of the representation matrices of algebra homomorphisms.
For an ideal $I$, let $\rad{I}$ denote its radical:
$\rad{I} = \set{ f \given \text{$f^m \in I$ for some $m \geq 1$} }$.
The precise setting that we consider is summarized as follows.

\begin{prop}\label{prop:Minors}
    Let $\Lambda$ and $\Gamma$ be finite-dimensional nilpotent algebras over a field $F$.
    Let $\cA$ be a filtered basis of $\Lambda$ and $\cB$ a filtered basis of $\Gamma$.
    Let $(\sigma_{xy}^z)$ and $(\tau_{ab}^c)$ be the structure constants of $\Lambda$ and $\Gamma$ with respect to $\cA$ and $\cB$, respectively.
    Define an ideal $I$ of a polynomial algebra $F[\, \upalpha_{ab} \mid a \in \cA, b \in \cB \,]$ by
    \begin{equation*}
        I = \gen[\Bigg]{ \sum_{z \in \cA} \sigma_{xy}^z \upalpha_{zc} - \sum_{a, b \in \cB} \tau_{ab}^c \upalpha_{xa} \upalpha_{yb} \given x, y \in \cA, c \in \cB }.
    \end{equation*}

    \begin{enumerate}
        \item If $|\cA_1| < |\cB_1|$ or $f \in \rad{I}$ for all maximal minors $f$ of the matrix $[\upalpha_{ab}]_{a \in \cA_1, b \in \cB_1}$, then $E \otimes_F \Gamma$ is not a homomorphic image of $E \otimes_F \Lambda$ for all field extensions $E/F$. \label{item:first}

        \item If $|\cA_1| \geq |\cB_1|$ and $f \notin \rad{I}$ for some maximal minor $f$ of the matrix $[\upalpha_{ab}]_{a \in \cA_1, b \in \cB_1}$, then $E \otimes_F \Gamma$ is a homomorphic image of $E \otimes_F \Lambda$ for some finite field extension $E/F$. \label{item:second}
    \end{enumerate}
\end{prop}
\index{$E$ : Field}

We need two lemmas before proving this proposition.
The facts of the lemma are certainly well known and several similar results appear in the literature, cf.~for example \cite[(5.2) Proposition]{Kuelshammer91} or \cite[Chapter III, Theorem 1.9(a)]{AuslanderReitenSmalo95}.
We did not find, however, the exact formulation we need and which has the same assumptions as we have here, so that we include a proof for completeness.

\begin{lem}\label{lem:RadicalSquare}
    Let $\Lambda$ and $\Gamma$ be finite-dimensional nilpotent algebras over a field $F$.
    Then an algebra homomorphism $\Lambda \to \Gamma$ is surjective if and only if the induced linear map $\Lambda/\Lambda^2 \to \Gamma/\Gamma^2$ is.
    Moreover, the size of a minimal generating set of $\Lambda$ as an $F$-algebra is the $F$-dimension of $\Lambda/\Lambda^2$.
    So, if $\cA$ is a filtered basis of $\Lambda$, then the size of a minimal generating set is $|\cA_1|$.
\end{lem}
\begin{proof}
    We first prove the following claim:
    if $\Pi$ is a subalgebra of $\Gamma$ such that $\Gamma = \Pi + \Gamma^2$, then $\Pi = \Gamma$.
    Indeed, $\Gamma = \Pi + \Gamma^2$ implies $\Gamma^n = \Pi^n + \Gamma^{n+1}$ for every positive integer $n$.
    Hence, \[\Gamma = \Pi + \Gamma^2 = \Pi + \Pi^2 + \Gamma^3 = \Pi + \Gamma^3 = \dotsb = \Pi + \Gamma^n \] for every $n$.
    As $\Gamma$ is nilpotent, the claim follows.

    Now let $\psi \colon \Lambda \to \Gamma$ be an algebra homomorphism and denote by $\psi^* \colon \Lambda/\Lambda^2 \to \Gamma/\Gamma^2$ the induced linear map.
    Clearly, if $\psi$ is surjective then so is $\psi^*$.

    On the other hand, if $\psi^*$ is surjective then, for each $b \in \Gamma$, there exists $a \in \Lambda$ such that $b + \Gamma^2 = \psi^*(a + \Lambda^2) = \psi(a) + \Gamma^2$.
    Hence $\Gamma = \psi(\Lambda) + \Gamma^2$ and $\psi$ is surjective by the claim for $\Pi = \psi(\Lambda)$.

    Let $\cA_1$ be a subset of $\Lambda$ whose natural images form a basis of $\Lambda/\Lambda^2$ as an $F$-vector space and similarly let $\cB_1$ be a subset of $\Gamma$ whose natural images in $\Gamma/\Gamma$ form a basis.
    For the second statement of the lemma note first that the minimal number of generators of $\Lambda$ is greater than or equal to the $F$-dimension of $\Lambda/\Lambda^2$, as the images of a generating set also generate $\Lambda/\Lambda^2$ and the product of any two elements in $\Lambda/\Lambda^2$ equals zero.
    On the other hand the $F$-dimension of $\Lambda/\Lambda^2$ equals $|\cA_1|$.
    As $\cA_1$ generates the whole algebra $\Lambda$ by the claim at the beginning of the proof, the minimal number of generators of $\Lambda$ is hence bounded from above by the $F$-dimension of $\Lambda/\Lambda^2$.
\end{proof}
\index{$\Pi$ : Algebra}
\index{$\psi^*$ : Linear Map}

Recall that for a commutative unital ring $R$, a commutative unital $R$-algebra $S$ is called faithfully flat over $R$, if a sequence $\mathcal{X}$ of $R$-modules is exact if and only if $S \otimes_R \mathcal{X}$ is exact, cf.~\cite[p.~45]{Matsumura86}.
\index{$S$ : Commutative Algebra}
\index{$X$ @ $\mathcal{X}$ : Sequence}
So, if $K/F$ is a field extension, then $K$ is faithfully flat over $F$.
By \cite[p.~46, \textit{Change of coefficient ring}]{Matsumura86}, if $S$ is a commutative unital $F$-algebra, then $KS$ is faithfully flat over $S$.
The following lemma is then a consequence of \cite[Theorem 7.5(ii)]{Matsumura86}.
\index{$K$ : Field}

\begin{lem}\label{lem:IdelaIntersection}
    Let $K/F$ be a field extension, $S$ a commutative unital $F$-algebra and $I$ an ideal of $S$.
    Then $KI \cap S = I$.
\end{lem}
Alternatively, it is not hard to prove the lemma by elementary linear algebra, using that a subset of an $F$-vector space $V$ is linearly independent if and only if it is linearly independent in $KV$.
We are now ready to prove the main proposition of this section.
\index{$V$ : Vector Space}

\begin{proof}[Proof of \cref{prop:Minors}]
    For a field extension $E/F$, write $\Lambda_E = E \otimes_F \Lambda$ and $\Gamma_E = E \otimes_F \Gamma$ for brevity.

    \eqref{item:first}
    If $|\cA_1| < |\cB_1|$, then by \cref{lem:RadicalSquare} the minimal number of generators of $\Lambda$ is smaller than the minimal number of generators of $\Gamma$.
    As these numbers remain unchanged under field extensions, it is clear that no surjective homomorphism from $\Lambda_E$ to $\Gamma_E$ can exist for any field extension $E/F$.

    Now assume that $|\cA_1| \geq |\cB_1|$ and $f \in \rad{I}$ for all maximal minors $f$ of the matrix $[\alpha_{ab}]_{a \in \cA_1, b \in \cB_1}$.
    We argue by contradiction.
    Assume that there is a surjective algebra homomorphism $\psi \colon \Lambda_E \to \Gamma_E$ for some field extension $E/F$.
    Then the representation matrix $[\alpha_{ab}]_{a \in \cA, b \in \cB}$ of $\psi$ with respect to $\cA$ and $\cB$ is defined by
    \begin{equation*}
        \psi(a) = \sum_{b \in \cB} \alpha_{ab} b
        \qquad
        (a \in \cA).
    \end{equation*}
    We simply write $F[\upalpha_{a b}]$ for the polynomial algebra $F[\, \upalpha_{a b} \mid a \in \cA, b \in \cB \,]$ (the upright letters are used for corresponding variables).
    Since $\psi$ is an algebra homomorphism,
    \begin{equation*}
        \psi(xy)       = \sum_{c \in \cB}\left(\sum_{z \in \cA} \sigma_{xy}^z \alpha_{zc}\right)c
        \quad\text{and}\quad
        \psi(x)\psi(y) = \sum_{c \in \cB}\left(\sum_{a, b \in \cB} \tau_{ab}^c \alpha_{xa} \alpha_{yb}\right)c
    \end{equation*}
    coincide for all $x, y \in \cA$.
    In other words, $(\alpha_{ab})$ is an $E$-solution of the simultaneous equations
    \begin{equation*}
        \sum_{z \in \cA} \sigma_{xy}^z \upalpha_{zc} - \sum_{a, b \in \cB} \tau_{ab}^c \upalpha_{xa} \upalpha_{yb} = 0 \\
        \qquad
        (x, y \in \cA, c \in \cB)
    \end{equation*}
    defined over $F$.
    So \[ \eta \colon F[\upalpha_{a b}]/I \to E,\ \upalpha_{ab} + I \mapsto \alpha_{ab} \] is a unital $F$-algebra homomorphism, which corresponds to $\psi$ through a canonical bijection of sets
    \begin{equation*}
        \Hom_E(\Lambda_E, \Gamma_E) \cong \Hom_F(F[\upalpha_{ab}]/I, E).
    \end{equation*}
    By \cref{lem:RadicalSquare}, the induced linear map $\psi^* \colon \Lambda_E/\Lambda_E^2 \to \Gamma_E/\Gamma^2_E$ is also surjective.
    Note that the dimension of $\Lambda_E/\Lambda_E^2$ equals the size of $\cA_1$ and the dimension of $\Gamma_E/\Gamma^2_E$ the size of $\cB_1$, also by \cref{lem:RadicalSquare}.
    Since $f \in \rad{I}$ for all maximal minors $f$ of $[\upalpha_{ab}]_{a \in \cA_1, b \in \cB_1}$, the matrix $[\alpha_{ab}]_{a \in \cA_1, b \in \cB_1}$ does not have full rank.
    As $|\cA_1| \geq |\cB_1|$ holds by assumption, this implies that the column vectors of $[\alpha_{ab}]_{a \in \cA_1, b \in \cB_1}$ are linearly dependent.
    Hence the induced linear map $\psi^* \colon \Lambda_E/\Lambda^2_E \to \Gamma_E/\Gamma^2_E$ is not surjective, a contradiction.

    \eqref{item:second}
    We argue by contraposition.
    Assume that $|\cA_1| \geq |\cB_1|$ and for no finite extension of fields $E/F$ there exists a surjective homomorphism $\Lambda_E \to \Gamma_E$.
    Let $\overline{F}$ be an algebraic closure of $F$.
    We first observe that neither does a surjective homomorphism $\Lambda_{\overline{F}} \to \Gamma_{\overline{F}}$ exist:
    indeed, if $\psi$ is such a surjective homomorphism with representation matrix $[\alpha_{ab}]_{a \in \cA, b \in \cB}$, then each entry is algebraic over $F$ and hence the field $E = F(\, \alpha_{ab} \mid a \in \cA, b \in \cB \,)$ is a finite extension of $F$.
    But then $\psi$ defines a surjective homomorphism $\Lambda_E \to \Gamma_E$, in contradiction with our assumption.

    Let $f$ be a maximal minor of the matrix $[\upalpha_{ab}]_{a \in \cA_1, b \in \cB_1}$.
    As no surjective homomorphism $\Lambda_{\overline{F}} \to \Gamma_{\overline{F}}$ exists, if $[\alpha_{ab}]_{a \in \cA, b \in \cB}$ is the representation matrix of some homomorphism $\Lambda_{\overline{F}} \to \Gamma_{\overline{F}}$, then by \cref{lem:RadicalSquare} the column vectors of the matrix $[\alpha_{ab}]_{a \in \cA_1, b \in \cB_1}$ are linearly dependent.
    As $|\cA_1| \geq |\cB_1|$ by assumption, this implies that $[\alpha_{ab}]_{a \in \cA_1, b \in \cB_1}$ does not have full rank.
    By a correspondence
    \begin{equation*}
        \Hom_{\overline{F}}(\Lambda_{\overline{F}}, \Gamma_{\overline{F}}) \cong \Hom_{\overline{F}}(\overline{F}[\upalpha_{ab}]/\overline{F}I, \overline{F}),
    \end{equation*}
    the maximal minor $f$ vanishes at every common zero of $\overline{F}
        I$.
    Hence, Hilbert's Strong Nullstellensatz \cite[Chapter VII, Theorem 14]{ZariskiSamuel} implies $f^m \in \overline{F}I$ for some $m \geq 1$.
    As $\overline{F}I \cap F[\upalpha_{a b}] = I$ by \cref{lem:IdelaIntersection}, we conclude $f^m \in I$.
\end{proof}
\index{$\eta$ : Algebra Homomorphism}

The above result may be presented as the following procedure.

\begin{procedure}[H]
    \Input{finite-dimensional nilpotent algebras $\Lambda$ and $\Gamma$ over a field $F$}
    \Output{$E \otimes_F \Gamma$ is not a homomorphic image of $E \otimes_F \Lambda$ for all field extensions $E/F$ or\\  $E \otimes_F \Gamma$ is a homomorphic image of $E \otimes_F \Lambda$ for some finite extension $E/F$}
    \BlankLine
    $\cA \gets \text{a filtered basis of $\Lambda$}$\;
    $\cB \gets \text{a filtered basis of $\Gamma$}$\;
    $(\sigma_{x y}^z) \gets \text{the structure constants of $\Lambda$ with respect to $\cA$}$\;
    $(\tau_{a b}^c) \gets \text{the structure constants of $\Gamma$ with respect to $\cB$}$\;
    $I \gets \gen{ \sum_{z \in \cA} \sigma_{xy}^z \upalpha_{zc} - \sum_{a, b \in \cB} \tau_{ab}^c \upalpha_{xa} \upalpha_{yb} \given x, y \in \cA, c \in \cB }$\;
    \eIf{$|\cA_1| < |\cB_1|$ or $f \in \rad{I}$ for all maximal minors $f$ of the matrix $[\upalpha_{a b}]_{a \in \cA_1, b \in \cB_1}$}{
        \Return{$E \otimes_F \Gamma$ is not a homomorphic image of $E \otimes_F \Lambda$ for all field extensions $E/F$}\;
    }{
        \Return{$E \otimes_F \Gamma$ is a homomorphic image of $E \otimes_F \Lambda$ for some finite extension $E/F$}\;
    }
    \caption{NaiveGeometricTest($\Lambda$, $\Gamma$)}
\end{procedure}

\index{$f$ : Polynomial}
This procedure terminates for a field $F$ over which there is an algorithm to solve the radical membership problem for polynomial algebras.
There are algorithms and implementations for the problem based on Gr\"obner bases for the prime field cases, which are the most relevant to us.

\begin{proof}[Correctness of \NaiveGeometricTest]
    It is clear from \cref{prop:Minors}.
\end{proof}

\vspace*{.5cm}
We now come to the more precise procedure in the context of the Modular Isomorphism Problem.
For a finite $p$-group $G$ and $F$ a field of characteristic $p$, the augmentation ideal $\Delta(FG)$ of $FG$ is a nilpotent algebra and it is well known that $FG \cong FH$ as unital $F$-algebras if and only if $\Delta(FG) \cong \Delta(FH)$ as $F$-algebras, cf.~\cite[2.~Corollary (a)]{Eick08}.
Let $\FF_q$ denote the field with $q$ elements.
\index{$q$ : Prime Power}
\index{$F_q$ @ $\FF_q$ : Field}
Hence, we may apply \NaiveGeometricTest to $\Delta(\FF_p G)$ and $\Delta(\FF_p H)$ to approach the Modular Isomorphism Problem, but it is possible to reduce computations further.
A first idea consists in considering $\Delta(\FF_p Q)$ instead of $\Delta(\FF_p H)$ for a homomorphic image $Q$ of $H$ that is not a homomorphic image of $G$.
A second idea consists in working with quotients of the algebras rather than the whole algebras;
usually, it is enough to work with much smaller algebras.
These ideas may be presented as follows.

\begin{procedure}[H]
    \Input{finite $p$-groups $G$ and $H$ of the same order}
    \Output{$FG \not\cong FH$ for all fields $F$ of characteristic $p$ or\\ $KG \cong KH$ for some finite field $K$ of characteristic $p$}
    \BlankLine
    \For{\text{a homomorphic image $Q$ of $H$ that is not a homomorphic image of $G$}}{
        $t \gets \min\set{ n \geq 1 \given \Delta^n(\FF_pQ) = 0 }$\;
        \For{$\ell \gets 1$ \KwTo $t-1$}{
            $\Lambda \gets \Delta(\FF_pG)/\Delta^{\ell+1}(\FF_pG)$\;
            $\Gamma \gets \Delta(\FF_pQ)/\Delta^{\ell+1}(\FF_pQ)$\;
            \If{$\NaiveGeometricTest(\Lambda, \Gamma)$ returns $F \otimes_{\FF_p} \Gamma$ is not a homomorphic image of $F \otimes_{\FF_p} \Lambda$ for all field extensions $F/\FF_p$}{
                \Return{$FG \not\cong FH$ for all fields $F$ of characteristic $p$}\;
            }
        }
    }
    \Return{$KG \cong KH$ for some finite field $K$ of characteristic $p$}\;
    \caption{RefinedGeometricTest($G$, $H$)}
\end{procedure}
\begin{proof}[Correctness of \RefinedGeometricTest]
    Suppose that the output is not ``$KG \cong KH$ for some finite field $K$ of characteristic $p$'', so $\NaiveGeometricTest(\Lambda, \Gamma)$ has returned that no surjective homomorphism exists from $F \otimes_{\FF_p} \Lambda$ to $F \otimes_{\FF_p} \Gamma$ for all field extensions $F/\FF_p$.
    Here $\Lambda = \Delta(\FF_pG)/\Delta^{\ell+1}(\FF_pG)$ and $\Gamma = \Delta(\FF_pQ)/\Delta^{\ell+1}(\FF_pQ)$ for some $\ell$ and some $Q$.
    Let $F$ be a field of characteristic $p$.
    It is clear that $\Delta(FQ)$ is a homomorphic image of $\Delta(FH)$, while $\Delta(FQ)$ is not a homomorphic image of $\Delta(FG)$ from the procedure.
    Hence $FG \not\cong FH$.

    On the other hand, assume the output is ``$KG \cong KH$ for some finite field $K$ of characteristic $p$''.
    If $G$ and $H$ are isomorphic, then of course $KG \cong KH$ holds for every field $K$.
    In case $G$ and $H$ are not isomorphic, the output means that $\NaiveGeometricTest(\Lambda, \Gamma)$ was applied to the values $\Lambda = \Delta(\FF_p G)/\Delta^{\ell + 1}(\FF_p G)$ and $\Gamma = \Delta(\FF_p H)$, for $\ell$ such that $\Delta(\FF_p H)^{\ell + 1} = 0$, and returned that there exists a surjective homomorphism from $K \otimes_{\FF_p} \Lambda$ to $K \otimes_{\FF_p} \Gamma$ for some finite field $K$ of characteristic $p$.
    In particular, the $F$-dimension of $\Lambda$ is bigger than or equal to the $F$-dimension of $\Gamma$.
    The latter equals $|H|-1$ and the dimension of $\Lambda$ can reach this value only if $\Delta(\FF_p G)^{\ell+1} = 0$, as the orders of $G$ and $H$ are equal by assumption.
    So, in this case there exists a surjective homomorphism $\Delta(K G) \to \Delta(KH)$ and as these algebras have the same finite $K$-dimension it must be an isomorphism.
    We conclude that in this case $KG \cong KH$.
\end{proof}

Note that \RefinedGeometricTest is asymmetric on inputs $G$ and $H$.
In practice a smaller ideal than the $I$ appearing in \NaiveGeometricTest can be chosen wisely by taking only distinctive relations into account as in \eqref{eq:IdealD8Q8}.
We comment on our actual implementation further in \Cref{sec:ComputationsOrder64}.

As a by-product, we obtain another proof of the reduction theorem by Garc\'ia-Lucas and del R\'io \cite[Corollary B]{GarciaLucasDelRio24} from \cref{prop:Minors}.
\begin{cor}\label{cor:Reduction}
    Let $G$ and $H$ be finite $p$-groups and $E$ a field of characteristic $p > 0$.
    If $EG \cong EH$ as unital algebras, then $KG \cong KH$ as unital algebras for some finite field $K$ of characteristic $p$.
\end{cor}

\section{Illustration: groups of order 8}\label{sec:Order8}
This section is devoted to illustrating the geometric approach described in the previous section with the first non-trivial example.
Namely, the groups we are going to work with are the dihedral group $G = \Dih{8}$ and the quaternion group $H = \Quat{8}$ of order~$8$ with presentations
\begin{alignat*}{8}
     & \Dih{8}  & \  & = & \ \langle\, x, y, z \mid x^2 & = 1, & \ y^2 & = 1, & \ z^2 & = 1, & \ [y, x] & = z, & \ [z, x] & = 1, & \ [z, y] & = 1 \,\rangle, \\
     & \Quat{8} & \  & = & \ \langle\, a, b, c \mid a^2 & = c, & \ b^2 & = c, & \ c^2 & = 1, & \ [b, a] & = c, & \ [c, a] & = 1, & \ [c, b] & = 1 \,\rangle.
\end{alignat*}
\index{$D_8$ @ $\Dih{8}$ : Finite Group}
\index{$Q_8$ @ $\Quat{8}$ : Finite Group}
\index{$x$ : Group Element}
\index{$y$ : Group Element}
\index{$z$ : Group Element}
\index{$a$ : Group Element}
\index{$b$ : Group Element}
\index{$c$ : Group Element}
Let $R$ be a commutative unital ring in which $2$ is not invertible.
Take a maximal ideal $\mathfrak{m}$ of $R$ containing $2$ and set $F = R/\mathfrak{m}$.
Also set $X = x - 1$, $Y = y - 1$ in $\Delta(FG)$ and $A = a - 1$, $B = b - 1$, $C = c - 1$ in $\Delta(FH)$.
\index{$X$ : Algebra Element}
\index{$Y$ : Algebra Element}
\index{$A$ : Algebra Element}
\index{$B$ : Algebra Element}
\index{$C$ : Algebra Element}
Then it can be verified that $\set{ A^i B^j C^k \given 0 \leq i, j, k < 2 }$ forms a basis of $FH$ over $F$, which is often called a Jennings basis.
\index{$i$ : Nonnegative Integer}
\index{$j$ : Nonnegative Integer}
\index{$k$ : Nonnegative Integer}
Note that $1$ is an element of this basis.
We could in fact work with $H$ as a basis, but it will become clear that this is less convenient, as it does not respect the filtration by powers of the augmentation ideal.
In the sense of the previous section we will work with the nilpotent algebras $\Lambda = \Delta(FG)/\Delta^4(FG)$ and $\Gamma = \Delta(FH)/\Delta^4(FH)$ and show that no surjective homomorphism $\Lambda \to \Gamma$ exists.

Let us assume that there exists an isomorphism of unital algebras $RG \to RH$ to get a contradiction.
Then we can obtain a normalized isomorphism $\psi\colon FG \to FH$.
From the adjunction $\Hom(FG, FH) \cong \Hom(G, (FH)^\times)$, the homomorphism $\psi$ comes from a group homomorphism $G \to (FH)^\times$, where $(FH)^\times$ denotes the group of units in $FH$.
Since the group $G$ is generated by $x$ and $y$, the values on these two points in $(FH)^\times$ completely determine the homomorphism.
This is equivalent to determining the values on $X$ and $Y$ in $\Delta(FH)$, as $\psi$ is normalized.
Let
\begin{equation*}
    \begin{aligned}
        \psi(X) & \equiv \alpha A + \beta B + \kappa AB + \lambda C, \\
        \psi(Y) & \equiv \gamma A + \delta B + \mu AB + \nu C
    \end{aligned}
    \mod \Delta^3(FH)
\end{equation*}
where $\alpha, \beta, \gamma, \delta, \kappa, \lambda, \mu, \nu \in F$.
\index{$\alpha$ : Scalar}
\index{$\beta$ : Scalar}
\index{$\gamma$ : Scalar}
\index{$\delta$ : Scalar}
\index{$\kappa$ : Scalar}
\index{$\lambda$ : Scalar}
\index{$\mu$ : Scalar}
\index{$\nu$ : Scalar}
Note here that $C \in \Delta^2(FH)$, cf.~\cite[Lemma 2.5]{MargolisSakuraiJMSJ}.
As $\psi$ respects multiplication, these coefficients must satisfy certain conditions.
From the first two defining relations of $G$, we have $X^2 = Y^2 = 0$.
We note that the following congruences hold:
\begin{equation*}
    \begin{aligned}
        BA  & \equiv AB + C + AC + BC, \\
        A^2 & \equiv B^2 \equiv C
    \end{aligned}
    \mod \Delta^4(FH).
\end{equation*}
Using these, direct calculations show that
\begin{equation*}
    \begin{aligned}
        \psi(X)^2 & \equiv (\alpha^2 + \alpha\beta + \beta^2)C + (\alpha\beta + \alpha\kappa)AC + (\alpha\beta + \beta\kappa)BC, \\
        \psi(Y)^2 & \equiv (\gamma^2 + \gamma\delta + \delta^2)C + (\gamma\delta + \gamma\mu)AC + (\gamma\delta + \delta\mu)BC
    \end{aligned}
    \mod \Delta^4(FH).
\end{equation*}
Let $I$ be an ideal in $\FF_2[\upalpha, \upbeta, \upgamma, \updelta, \upkappa, \uplambda, \upmu, \upnu]$ generated by these coefficients (the upright letters are used for corresponding variables), that is,
\begin{equation}\label{eq:IdealD8Q8}
    I = \langle \upalpha^2 + \upalpha\upbeta + \upbeta^2,\ \upalpha\upbeta + \upalpha\upkappa,\ \upalpha\upbeta + \upbeta\upkappa,\ \upgamma^2 + \upgamma\updelta + \updelta^2,\ \upgamma\updelta + \upgamma\upmu,\ \upgamma\updelta + \updelta\upmu \rangle.
\end{equation}
\index{$\alpha$ @ $\upalpha$ : Variable}
\index{$\beta$ @ $\upbeta$ : Variable}
\index{$\gamma$ @ $\upgamma$ : Variable}
\index{$\delta$ @ $\updelta$ : Variable}
\index{$\kappa$ @ $\upkappa$ : Variable}
\index{$\lambda$ @ $\uplambda$ : Variable}
\index{$\mu$ @ $\upmu$ : Variable}
\index{$\nu$ @ $\upnu$ : Variable}
Then, with the help of a computer algebra system, such as for example \texttt{Oscar}, we can easily get
\begin{equation*}
    \begin{vmatrix} \upalpha & \upbeta \\ \upgamma & \updelta \end{vmatrix} \in \rad{I},
\end{equation*}
or with an explicit exponent and coefficients,
\begin{align*}
    \begin{vmatrix} \upalpha & \upbeta \\ \upgamma & \updelta \end{vmatrix}^4
     & =
    \upbeta^2\upgamma^4(\upalpha^2 + \upalpha\upbeta + \upbeta^2)
    + \upbeta^2\upgamma^4(\upalpha\upbeta + \upalpha\upkappa)
    + \upalpha\upbeta\upgamma^4(\upalpha\upbeta + \upbeta\upkappa) \\
     & \quad
    + \upalpha^4\updelta^2(\upgamma^2 + \upgamma\updelta + \updelta^2)
    + \upalpha^4\updelta^2(\upgamma\updelta + \upgamma\upmu)
    + \upalpha^4\upgamma\updelta(\upgamma\updelta + \updelta\upmu)
\end{align*}
using the functions \texttt{radical\_membership} or \texttt{coordinates}.
In particular, the determinant $\begin{vmatrix} \alpha & \beta \\ \gamma & \delta \end{vmatrix}$ must vanish.

On the other hand, as $\psi$ is an isomorphism, the induced linear map $\Delta(FG)/\Delta^2(FG) \to \Delta(FH)/\Delta^2(FH)$ should be an isomorphism.
Hence its determinant must be invertible and nonzero, which contradicts the facts we have seen in the previous paragraph.
Therefore, we establish $RG \not\cong RH$.

\begin{rmk}
    In the case $F = \FF_2$, proving $FG \not\cong FH$ is easier than the general case.
    This can be seen from the geometric viewpoint as follows.
    First, it is enough to work with $\psi(X) \equiv \alpha A + \beta B \mod \Delta^2(FH)$.
    Second, it is enough to use only the first relation $X^2 = 0$.
    Then $\psi(X)^2 \equiv (\alpha^2 + \alpha\beta + \beta^2)C \mod \Delta^3(FH)$.
    Since $F = \FF_2$, the coefficients also satisfy $\alpha^2 + \alpha = \beta^2 + \beta = 0$.
    As before, let
    \begin{equation*}
        I = \langle \upalpha^2 + \upalpha\upbeta + \upbeta^2,\ \upalpha^2 + \upalpha,\ \upbeta^2 + \upbeta \rangle
    \end{equation*}
    be an ideal in $\FF_2[\upalpha, \upbeta]$ and observe that $\upalpha, \upbeta \in I$ as
    \begin{align*}
        \upalpha & = \upalpha(\upalpha^2 + \upalpha\upbeta + \upbeta^2) + (\upalpha + \upbeta + 1)(\upalpha^2 + \upalpha) + \upalpha(\upbeta^2 + \upbeta), \\
        \upbeta  & = \upbeta(\upalpha^2 + \upalpha\upbeta + \upbeta^2) + \upbeta(\upalpha^2 + \upalpha) + (\upalpha + \upbeta + 1)(\upbeta^2 + \upbeta).
    \end{align*}
    \index{$I$ : Ideal}
    This shows that $\psi(X) \in \Delta^2(FH)$, a contradiction.
    In a more complicated setting, a somewhat similar argument is already used in the proof of Proposition 3.2 in \cite{MargolisSakuraiJMSJ}, for example.
    \fxnote{
        Check how the proposition is numbered in the published version at the end.
    }
    In this sense, the geometric approach presented in this paper can be seen as an expanded formulation of a group base approximation used in our previous studies \cite{MargolisSakurai25,MargolisSakuraiJMSJ}.

    The fact that
    \begin{equation*}
        \Delta(\FF_2\Dih{8})/\Delta^3(\FF_2\Dih{8}) \not\cong \Delta(\FF_2\Quat{8})/\Delta^3(\FF_2\Quat{8})
        \quad\text{while}\quad
        \Delta(\FF_4\Dih{8})/\Delta^3(\FF_4\Dih{8}) \cong \Delta(\FF_4\Quat{8})/\Delta^3(\FF_4\Quat{8})
    \end{equation*}
    is also explicitly calculated in \cite[Example 2.11]{MargolisSakuraiSaoPaulo}.
\end{rmk}

\begin{rmk}
    At first glance, one might hope that, in the context of the MIP, if $G$ and $H$ are groups such that $FG \not\cong FH$ and $Q$ is a quotient of $H$ which is not isomorphic to a quotient of $G$, then there is no surjective homomorphism $\Delta(FG) \to \Delta(FQ)$.

    This is however not the case, as can be seen from the quotients of the known negative solutions to the MIP given in detail in the appendix of \cite{MargolisSakurai25}.
    One can also easily check this using computer algebra for the smallest negative solutions.
    These are the groups $G$ and $H$ of order $512$ with group identifiers 453 and 456 in the Small Groups Library of GAP (given explicitly in \cite{GarciaLucasMargolisDelRio22}).
    Each of them has a quotient of order $64$ which is not a quotient of the other group.
    But of course the isomorphism between the group algebras of $G$ and $H$ implies that each of them maps onto the group algebra of any quotient of the other.
\end{rmk}

\section{Groups of order 64}\label{sec:ComputationsOrder64}

The groups of order $64$ are the most classical class for computer algebraic investigations of the Modular Isomorphism Problem.
It was first studied by Wursthorn with the first implemented algorithm for the MIP \cite{Wursthorn93}.
He was able to solve the MIP over the prime field for all groups of order $64$ in this project.
Many more details are given in Wursthorn's Diplomarbeit \cite{Wursthorn90}.
These results were independently verified again by Eick when she implemented the second algorithm to study the MIP \cite{Eick08, MargolisMoede22}.

Note that there are exactly $267$ groups of order $64$ up to isomorphism, so that, if we would compare each pair of groups, we would have to consider $35511$ cases.
However, from the known theoretical results on the MIP and the properties of a finite $p$-group which are known to be invariant under isomorphisms of modular group algebras we are reduced to the $18$ cases which are investigated below.
Here we apply our algorithm to show that the MIP has a positive solution for groups of order $64$ over all fields of characteristic $2$.

\begin{prop}\label{prop:ord64}
    Let $F$ be a field of characteristic $2$ and let $G$ and $H$ be groups.
    If the order of $G$ divides $64$, then $FG \cong FH$ as unital algebras implies $G \cong H$ as groups.
\end{prop}
\begin{proof}
    For groups of order $32$ this is recorded in \cite[Theorem 2.1(5)]{MargolisSakuraiSaoPaulo} and for groups of order $16$ or less it follows also, as each such group $G$ is either metacyclic or satisfies $|G/\Z{G}| \leq 4$, so a positive solution also follows from \cite[Theorem 2.1]{MargolisSakuraiSaoPaulo}.

    To apply \RefinedGeometricTest to groups of order $64$, we first apply the command \texttt{MIPSplitGroups\-ByGroupTheoreticalInvariants\-AllFields} of \texttt{ModIsom}\footnote{Applying the command from version 3.1.0, there is a mistake in previous versions.
    } \cite{ModIsom} from the computer algebra system \texttt{GAP} \cite{GAP}.
    This results in a list of all groups of order $64$ for which the MIP over all fields remains open at the moment.
    It turns out that we are left with $18$ pairs of groups to investigate, the data on which known invariants can be used to reduce to these cases is given in \cref{sec:Appendix}.
    The groups involved and other relevant details of our computations are summarized in \cref{tbl:ComputationsGroups64}.
    Here the columns mean the following:

    \begin{description}
        \item[$G$, $H$] The group identifiers $r$ and $s$ we are considering, namely whether a surjective homomorphism from $\Delta(FG)$ to $\Delta(FH)$ exists for $G = \SmallGroup(64,r)$ and $H = \SmallGroup(64, s)$.
        \item[$Q$] We actually investigate whether there is a surjective homomorphism of $F$-algebras from $\Delta(FG)$ to $\Delta(FQ)$.
              The property of $Q$ is that there exists a surjective homomorphism from $H$ to $Q$ but not from $G$ to $Q$.
              The entry in the table gives the group identifier in the Small Groups Library.
              In particular, the first number is its order.
        \item[$\ell$] Least integer $\ell$ such that there is no surjective homomorphism from $\Delta(FG)/\Delta^{\ell+1}(FG)$ to $\Delta(FH)/\Delta^{\ell+1}(FH)$ for all fields $F$ of characteristic $2$.
        \item[$\ell(\FF_4)$] Least integer $\ell$ such that there is no surjective homomorphism from $\Delta(\FF_4G)/\Delta^{\ell+1}(\FF_4G)$ to $\Delta(\FF_4H)/\Delta^{\ell+1}(\FF_4H)$.
              This is computed by applying \texttt{MIP\-SplitGroups\-ByAlgebras} of \texttt{ModIsom} as \texttt{MIP\-SplitGroups\-ByAlgebras([G,H], 2)}.
              \footnote{For the last case (238, 239) the calculation in \texttt{ModIsom} does not finish in a reasonable time, so the integer is a consequence of our results.}
        \item[$\ell(\FF_2)$] Least integer $\ell$ such that there is no surjective homomorphism from $\Delta(\FF_2G)/\Delta^{\ell+1}(\FF_2G)$ to $\Delta(\FF_2H)/\Delta^{\ell+1}(\FF_2H)$.
              This is computed by applying \texttt{MIP\-SplitGroups\-ByAlgebras} of \texttt{ModIsom} to the input \texttt{[G,H]}.
        \item[$\dim \Lambda$] Dimension of $\Lambda = \Delta(FG)/\Delta^{\ell+1}(FG)$.
        \item[$\dim \Gamma$] Dimension of $\Gamma = \Delta(FQ)/\Delta^{\ell+1}(FQ)$.
              So we are showing that there is no surjective homomorphism from a certain algebra of dimension $\dim\Lambda$ to a certain algebra of dimension $\dim \Gamma$.
        \item[$\rank{G}$] Size of minimal generating set of $G$.
              The polynomial algebra in which the computation takes place has $\rank{G}\cdot \dim \Gamma$ variables.
        \item[$\rank{Q}$] Size of minimal generating set of $Q$.
              Hence, we are considering whether all the minors of a $\rank{G} \times \rank{Q}$-matrix lie in a certain ideal.
        \item[$\FF_2$-inv.] Whether some known $\FF_2$-invariants are able to distinguish the group algebras $\FF_2G$ and $\FF_2H$.
              We remark that every entry ``Known'' can be derived from an invariant connected to the small group algebra:
              for all but the first case (13,14) the invariant $G/\gamma_2(G)^p\gamma_3(G)$ is sufficient, i.e. in these cases $G/\gamma_2(G)^p\gamma_3(G) \not\cong H/\gamma_2(H)^p\gamma_3(H)$.
              Here $\gamma_n(G)$ denotes the $n$-th term of the lower central series of $G$.
              This is an invariant by the work of Sandling \cite{Sandling89}.
              For the pair (13,14) one can use the invariant $G/\gamma_2(G)^p\gamma_4(G)$, as the groups are $2$-generated.
              This is an invariant by \cite[Theorem 1.2]{MargolisMoede22}.
    \end{description}

    \index{$\gamma_n(G)$ : Lower Central Series}

    \begin{table}
        \caption{Results for the remaining groups of order 64.}
        \label{tbl:ComputationsGroups64}
        \begin{tabular}{cccccccccccc}
            \toprule
            $G$ & $H$ & $Q$      & $\ell$ & $\ell(\FF_4)$ & $\ell(\FF_2)$ & $\dim \Lambda$ & $\dim \Gamma$ & $\rank{G}$ & $\rank{Q}$ & $\FF_2$-inv. \\
            \midrule
            13  & 14  & (32, 7)  & 6      & 6             & 4             & 27             & 21            & 2          & 2          & Known        \\
            65  & 70  & (32, 25) & 2      & 2             & 2             & 9              & 8             & 3          & 3          & Known        \\
            105 & 104 & (16, 6)  & 4      & 4             & 4             & 23             & 8             & 3          & 2          & Known        \\
            142 & 155 & (8, 4)   & 3      & 3             & 2             & 15             & 6             & 3          & 2          & Known        \\
            142 & 157 & (8, 4)   & 3      & 3             & 2             & 15             & 6             & 3          & 2          & Known        \\
            155 & 157 & (16, 8)  & 7      & 7             & 7             & 47             & 14            & 3          & 2          & Unknown      \\
            156 & 158 & (16, 9)  & 7      & 7             & 7             & 47             & 14            & 3          & 2          & Unknown      \\
            160 & 156 & (16, 8)  & 7      & 7             & 7             & 47             & 14            & 3          & 2          & Unknown      \\
            160 & 158 & (16, 9)  & 7      & 7             & 7             & 47             & 14            & 3          & 2          & Unknown      \\
            167 & 173 & (32, 34) & 3      & 3             & 2             & 15             & 15            & 3          & 3          & Known        \\
            167 & 176 & (32, 34) & 3      & 3             & 2             & 15             & 15            & 3          & 3          & Known        \\
            176 & 173 & (16, 8)  & 7      & 7             & 7             & 47             & 14            & 3          & 2          & Unknown      \\
            168 & 179 & (8, 4)   & 3      & 3             & 2             & 15             & 6             & 3          & 2          & Known        \\
            168 & 180 & (8, 4)   & 3      & 3             & 2             & 15             & 6             & 3          & 2          & Known        \\
            180 & 179 & (16, 8)  & 7      & 7             & 7             & 47             & 14            & 3          & 2          & Unknown      \\
            172 & 182 & (8, 4)   & 3      & 3             & 2             & 15             & 6             & 3          & 2          & Known        \\
            175 & 181 & (8, 4)   & 3      & 3             & 2             & 15             & 6             & 3          & 2          & Known        \\
            239 & 238 & (16, 13) & 2      & 2             & 2             & 12             & 7             & 4          & 3          & Known        \\
            \bottomrule
        \end{tabular}
    \end{table}

    These calculations complete the proof.
\end{proof}
\index{$Z(G)$ @ $\Z{G}$ : Center}

The proof of the main theorem now follows easily:

\begin{proof}[Proof of \Cref{main:64}]
    Take a maximal ideal $\mathfrak{m}$ of $R$ containing $2$ and set $F = R/\mathfrak{m}$.
    If $RG \cong RH$ then
    \begin{equation*}
        FG \cong F \otimes_R RG \cong F \otimes_R RH \cong FH.
    \end{equation*}
    Hence $G \cong H$ follows from \cref{prop:ord64}.
\end{proof}

\vspace*{.5cm}
We note that our approach has some similarities to the study of the MIP for groups of order $64$ over $\FF_2$ as presented by Hertweck and Soriano in \cite{HertweckSoriano06}.
Also there, the idea is that if an isomorphism of unital algebras $\FF_2G \to \FF_2H$ exists, then $G$ embeds into a certain section of $\FF_2H$ and in fact maps onto a section of $\FF_2 Q$ where $Q$ is a group which is isomorphic to a quotient of $H$ but not of $G$.
Actually, in many cases for the same pair $G$ and $H$ their choice of $Q$ coincides with ours.
They also show, for instance, by manual calculations that $\ell(\FF_2)$ for the groups $\SmallGroup(64, 173)$ and $\SmallGroup(64, 176)$ is at least $7$.
The main difference with our approach, apart from the use of algorithms from algebraic geometry, is that they choose so-called Zassenhaus ideals instead of powers of the augmentation ideal.
This has the advantage that smaller sections can be studied, but the disadvantage that a key property of the Zassenhaus ideals is not known to hold over general fields, cf.~\cite[Section 2.2, Difference 2]{MargolisSakuraiSaoPaulo}.

We finish with some remarks on the implementation of our ideas, the code is available at \cite{GitRep}.
We choose the computer algebra package \texttt{Oscar} \cite{Oscar,OSCAR-book} from \texttt{Julia} for our implementation, as it allows us to combine procedures available in \texttt{GAP} and \texttt{Singular} \cite{Singular}.
Let $r = \rank{G}$ and $s = \dim \Delta(\FF_pQ)/\Delta^{\ell+1}(\FF_pQ)$.
\index{$r$ : Positive Integer}
\index{$s$ : Positive Integer}
To obtain the polynomials used to generate the ideal $I$ we use the polynomial algebra $R = \FF_p[\, \upalpha_{i j} \mid 1 \leq i \leq r, \ 1 \leq j \leq s \,]$ and the arithmetic for nilpotent algebras available in the \texttt{GAP}-package \texttt{ModIsom} \cite{ModIsom}.
To do this we formally define a table algebra, in the sense of the package \texttt{ModIsom}, corresponding to $\Delta(RQ)/\Delta^{\ell+1}(RQ)$ and think of the element of $\Delta(RG)$ corresponding to the $i$-th generator of $G$ being mapped in this algebra to $\upalpha_{i 1}b_1+\dotsb+\upalpha_{i s}b_s$, where $b_1,\dotsc,b_s$ denotes an $R$-basis of $\Delta(RQ)/\Delta^{\ell+1}(RQ)$.
\index{$\alpha_{i j}$ : Scalar}
\index{$b_1, \dotsc, b_s$ : Algebra Element}
We then define an ideal $I$ generated by the polynomials which one obtains as coefficients applying the defining relations of the group $G$ to a set of elements of $\Delta(RQ)/\Delta^{\ell+1}(RQ)$ to which they correspond.
These defining relations come from a presentation of $G$ as a \texttt{simplified\_fp\_group}, but in principle other presentations and more or less relations could be used.
We note that as the $\Lambda$ we actually work with in the sense of \RefinedGeometricTest is $\Delta(\FF_p G)/\Delta(\FF_p G)^{\ell + 1}$, using all the defining relations of $G$ defines an ideal which might be bigger than the minimal one possible.
However, this provides a unified way of working with all the pairs of groups studied above and in fact speeds up computations considerably in many cases.

The basis $\{b_1,\dotsc,b_s\}$ as implemented in \texttt{ModIsom} is automatically filtered, so that the first $\rank{Q}$ elements provide a basis of $\Delta(RQ)/\Delta^{2}(RQ)$.
We then define the $(\rank{G} \times \rank{Q})$-matrix $[\upalpha_{i j}]_{1\leq i \leq \rank{G}, \ 1 \leq j \leq \rank{Q}}$ and test whether all the maximal minors of the matrix lie in $\rad{I}$.
For this we use the command \texttt{radical\_membership}, coming from \texttt{Singular}.
Moreover, to obtain values for $Q$ and $\ell$ which might be promising we analyze the quotients of $H$ which are not quotients of $G$, choosing a minimal one among them, such that $FQ$ is not a homomorphic image of $FG$.
Note that the isomorphism type of this $Q$ is in general not unique, there might be several groups of the same order having this property.
Hence, a user running our code might obtain some other group in the column $Q$ (in the $18$ cases above this can happen only for the pair with ID's 65 and 70).
Furthermore, we obtain the smallest $\ell$ such that $\Delta(\FF_qG)/\Delta^{\ell+1}(\FF_qG)$ is not isomorphic to $\Delta(\FF_qH)/\Delta^{\ell+1}(\FF_qH)$, for $q \in \{p,p^2\}$, by using the main commands of \texttt{ModIsom} such as \texttt{MIP\-SplitGroups\-ByAlgebras}.

On the practical side, our calculations do not need too much memory, the most extreme case allocates around 7 GB in total and occurs for the ID's 239 and 238.
Most of the cases work very fast, only the case of the groups with ID's 180 and 179 and the case of the groups with ID's 176 and 173 need more than one minute.
The latter is the most time consuming case, but it also finishes in about three hours.

To finish this article we wish to highlight an open problem which can be attacked by our procedure, although it seems doubtful it can be solved in reasonable time using our implementation.
It has been also stated in \cite[Question]{GarciaLucasDelRio24}:

\begin{que}
    Do there exist finite $p$-groups $G$ and $H$ such that $\FF_p G \not\cong \FF_p H$ but $\FF_q G \cong \FF_q H$ for some power $q$ of $p$?
\end{que}

\appendix
\section{Data for groups of order 64}\label{sec:Appendix}

We provide here the data used to reduce the MIP over all fields for groups of order $64$ to the $18$ cases described in \cref{sec:ComputationsOrder64}.
Throughout, $F$ denotes a field of characteristic $2$.
Namely, we first determine, if a group lies in a class of groups for which the MIP is known to have a positive solution over all fields.
Then, for the remaining groups, we list invariants which are able to show that for any other pair of groups $G$ and $H$ of order $64$ which do not appear among the $18$ pairs in \cref{sec:ComputationsOrder64} the group algebras $FG$ and $FH$ are not isomorphic.
The theoretical results are all listed in \cite{MargolisSakuraiSaoPaulo} and we separate the other groups by various known group-theoretical invariants, some of them rather proceeding from group cohomology.
Although we do not study cohomology rings in details, we remark that for all groups which are left and which are handled by our new algorithm in \cref{sec:ComputationsOrder64}, the Poincar\'e series encoding the dimensions of cohomology groups for the trivial module cannot distinguish them by \cite[Appendix]{CarlsonTownsley03} (cf.~also \cite{GreenKing15}).
We will refer to \cite{MargolisSakuraiSaoPaulo} which contains a summary of known invariants over all fields, although it is of course not the original source of all these invariants.

We first summarize which of the groups of order $64$ lie in a class for which the MIP is known to hold over all fields.
We use \cite[Theorems 2.1, 3.6, 3.9]{MargolisSakuraiSaoPaulo} and include a group only in one of the classes, even when in lies in several of them.
We denote by $Z(G)$ the center of a group $G$ and by $\Phi(G)$ its Frattini subgroup.
The identifiers of the groups in the Small Groups Library of \texttt{GAP} for which we can apply one of the theoretical results are the following.
\begin{itemize}
    \item $|G/\Z{G}| \leq 4$  \cite[Theorem 2.1(1)]{MargolisSakuraiSaoPaulo}: \\
          \noindent
          1, 2, 3, 17, 26, 27, 29, 44, 50, 51, 55, 56, 57, 58, 59, 83, 84, 85, 86, 87, 103, 112, 115, 126, 183, 184, 185, 192, 193, 194, 195, 196, 197, 198, 246, 247, 248, 260, 261, 262, 263, 267.\\

    \item $G$ is $2$-generated of nilpotency class $2$ \cite[Theorem 3.9]{MargolisSakuraiSaoPaulo}: \\
          \noindent
          18, 19, 28.\\

    \item $G$ is metacyclic  \cite[Theorem 3.6]{MargolisSakuraiSaoPaulo}:\\
          \noindent
          15, 16, 45, 46, 47, 48, 49, 52, 53, 54.\\

    \item $G$ is of nilpotency class $3$ and $|G/\Z{G}| = |\Frat{G}| = 8$ \cite[Theorem 2.1(6)]{MargolisSakuraiSaoPaulo}: \\
          \noindent
          95, 96, 97, 101, 106, 107, 108, 110, 118, 119, 120, 124.\\

    \item $\Z{G}$ is cyclic and $G/\Z{G}$ is dihedral \cite[Theorem 2.1(7)]{MargolisSakuraiSaoPaulo}: \\
          \noindent
          31, 40, 189.\\

    \item $G$ is $2$-generated and $G/\Z{G}$ is non-abelian dihedral and not a known negative solution of the MIP \cite[Theorem 2.1(8)]{MargolisSakuraiSaoPaulo}: \\
          \noindent
          6, 7, 20, 21, 22, 38, 39.\\
\end{itemize}
\index{$\Phi(G)$ @ $\Frat{G}$ : Frattini Subgroup}

Next, we can handle those groups which have an elementary abelian direct factor, as these reduce to smaller groups by \cite[Proposition 2.6]{MargolisSakuraiSaoPaulo} and for $2$-groups of order at most $32$ the Modular Isomorphism Problem over all fields has a positive answer, as mentioned in the proof of \Cref{main:64} above.
The maximal elementary abelian direct factor can be obtained using the function \texttt{Maximal\-ElementaryAbelian\-DirectFactor} from \texttt{ModIsom}.
Among the groups not listed among those which fall into a class with a positive solution above, the following have a non-trivial elementary abelian direct factor.
Again we list their ID's, these are:
\\ \noindent 90, 92, 93, 186, 187, 188, 202, 203, 204, 205, 207, 208, 209, 211, 212, 250, 251, 252, 253, 254, 255, 264, 265.
\\

The data for the remaining groups is organized in Tables \labelcref{tbl:2Generated}, \labelcref{tbl:3Generated} and \labelcref{tbl:4Generated} which correspond to the $2$-, $3$- and $4$-generated groups, respectively.
Note that the minimal number of generators is a known invariant, as it equals the $\FF_p$-dimension of $G/\Frat{G} = \Jen_1(G)/\Jen_2(G)$, cf.~\cite[Chapter 14, Lemma 2.7]{Passman77}.
Here by $\Jen_n(G)$ we denote the $n$-th term of the Jennings series of $G$ which coincides with the dimension subgroup series.
\index{$D_n(G)$ @ $\Jen_n(G)$ : Jennings Series}
This organization of the data also seems natural, as it corresponds to the order of groups in the Small Groups Library.
We note that except for one group every $5$-generated group of order $64$ is already listed above, so a table for this class of groups is not necessary.

The columns in the tables mean the following, where we also indicate the \texttt{GAP}-functions from the package \texttt{ModIsom} which can be used to obtain them:
\begin{itemize}
    \item $G$ is the index of the group of order $64$ in the Small Groups Library.
    \item \textit{$\dim \HH^1$} stands for the dimension of the first Hochschild cohomology group which is
          \[\dim \HH^1(FG) = \sum_{g^G} \dim \rH^1(\Cent_G(g), F) = \sum_{g^G} \log_p |\Cent_G(g)/\Frat{\Cent_G(g)}|,\]
          sometimes called the Roggenkamp parameter, cf.~\cite[Proposition 2.8]{MargolisSakuraiSaoPaulo}.
          Here the sums run over all the conjugacy classes of $G$.
          This is the first entry of the output by the function \texttt{ConjugacyClassInfo}.
    \item \textit{K} stands for the number of conjugacy classes of $p^i$-th powers for $p^i = 0, p, \dotsc, \exp(G)$, see \cite[Proposition 2.7]{MargolisSakuraiSaoPaulo}.
          These are the second to penultimate entries of the output of \texttt{ConjugacyClassInfo}.
    \item \textit{Q} stands for the number of conjugacy classes of maximal elementary abelian subgroups of order $p, p^2, \dotsc, p^k$, where $p^k$ is the maximal order of an elementary abelian subgroup of $G$, see \cite[Proposition 2.8]{MargolisSakuraiSaoPaulo}.
          This is the output of \texttt{SubgroupsInfo}.
    \item \textit{G-L} gives the order of the section $\Jen_1(\tilde{G})/\Jen_2(\tilde{G})$, where
          \[ \tilde{G} = \Omega(G : G') = \gen{ g \in G \given g^p \in G' }. \]
          Here $G'$ denotes the derived subgroup of $G$.
          This is an invariant by an application of Lemma 3.2(1) to Example 3.7(2) in \cite{GarciaLucas24}.
          This can be obtained by taking the element \texttt{e} of the output of \texttt{NormalSubgroupsInfo} which contains \texttt{"Omega"} in its first entry and then dividing \texttt{e[2][1]} by \texttt{e[2][2]}.
    \item \textit{$\dim \rH^2$} is the dimension of the second cohomology group $\rH^2(G, F)$, where $F$ means the trivial $FG$-module.
          This can be obtained from the Poincar\'e series of the group given in \cite{CarlsonTownsley03} (cf.~also \cite{GreenKing15}).
          It also follows applying the function \texttt{DimensionTwoCohomology}.
    \item \textit{$\dim \HH^2$} is the dimension of the second Hochschild cohomology group which equals $\sum_{g^G} \dim \rH^2(\Cent_G(g), F)$, where the sum runs over the conjugacy classes of $G$.
          It can be obtained using \texttt{DimensionSecondHochschild}.
\end{itemize}
\index{$HH^1$ @ $\HH^1$ : Cohomology Group}
\index{$C_G(g)$ @ $\Cent_G(g)$ : Centralizer}
\index{$H^1$ @ $\rH^1$ : Cohomology Group}
\index{$G$ @ $\tilde{G}$ : Finite Group}
\index{$\Omega(G : G')$ : Relative Omega Subgroup}
\index{$H^2$ @ $\rH^2$ : Cohomology Group}
\index{$HH^2$ @ $\HH^2$ : Cohomology Group}

The group ID's which are bold, are those groups which survive all tests and are handled by our algorithm, as presented in \cref{sec:ComputationsOrder64}.

\newgeometry{left=1cm,bottom=0.1cm}
\begin{table}
    \begin{minipage}{.38\textwidth}    \caption{Data for the remaining 2-generated groups of order 64} \label{tbl:2Generated}
        \begin{tabular}{cccc}
            \toprule
            $G$         & $\dim \HH^1$ & K              & Q       \\ \midrule
            4           & 52           &                &         \\
            5           & 50           & 22, 7, 2, 1    &         \\
            8           & 44           &                &         \\
            9           & 43           &                &         \\
            10          & 39           &                &         \\
            11          & 38           & 19, 8, 3, 1    &         \\
            12          & 40           & 19, 6, 3, 1    &         \\
            \textbf{13} & 38           & 19, 6, 3, 1    &         \\
            \textbf{14} & 38           & 19, 6, 3, 1    &         \\
            23          & 62           &                &         \\
            24          & 54           &                &         \\
            25          & 50           & 22, 6, 2, 1    &         \\
            30          & 40           & 22, 8, 4, 2, 1 &         \\
            32          & 31           &                &         \\
            33          & 28           & 13, 5, 2, 1    &         \\
            34          & 28           & 13, 5, 1       &         \\
            35          & 27           & 13, 5, 1       &         \\
            36          & 24           &                &         \\
            37          & 23           &                &         \\
            41          & 29           &                & 0, 2    \\
            42          & 29           &                & 0, 0, 1 \\
            43          & 27           & 16, 6, 3, 2, 1 &         \\ \bottomrule
        \end{tabular}
    \end{minipage}
    \begin{minipage}{.45\textwidth}  \caption{Data for the remaining 4-generated groups of order 64} \label{tbl:4Generated}
        \begin{tabular}{ccccc}
            \toprule
            $G$          & $\dim \HH^1$ & K           & Q          & $\dim \HH^2$ \\ \midrule
            199          & 108          & 34, 4, 1    & 0, 0, 2    &              \\
            200          & 108          & 34, 4, 1    & 0, 0, 1    &              \\
            201          & 108          & 34, 4, 1    & 0, 0, 3    &              \\
            206          & 100          &             &            &              \\
            210          & 88           &             &            &              \\
            213          & 92           &             &            &              \\
            214          & 84           & 28, 4, 1    &            &              \\
            215          & 78           &             &            &              \\
            216          & 74           & 22, 4, 1    &            &              \\
            217          & 70           & 22, 4, 1    &            &              \\
            218          & 74           & 22, 3, 1    &            &              \\
            219          & 68           &             & 0, 0, 2, 1 &              \\
            220          & 66           &             &            &              \\
            221          & 68           &             & 0, 0, 5    &              \\
            222          & 60           &             &            &              \\
            223          & 64           &             &            &              \\
            224          & 70           & 22, 3, 1    &            &              \\
            225          & 62           &             &            &              \\
            226          & 90           &             &            &              \\
            227          & 84           & 25, 4, 1    &            &              \\
            228          & 82           &             &            &              \\
            229          & 80           &             & 0, 0, 4    &              \\
            230          & 80           &             & 0, 0, 2    &              \\
            231          & 80           &             & 0, 0, 6    &              \\
            232          & 77           &             &            &              \\
            233          & 74           & 25, 4, 1    & 0, 0, 2    &              \\
            234          & 76           &             & 0, 0, 4    &              \\
            235          & 76           &             & 0, 0, 2    &              \\
            236          & 74           & 25, 4, 1    & 0, 0, 4    & 135          \\
            237          & 72           &             &            &              \\
            \textbf{238} & 70           & 25, 4, 1    &            &              \\
            \textbf{239} & 70           & 25, 4, 1    &            &              \\
            240          & 74           & 25, 4, 1    & 0, 0, 4    & 139          \\
            241          & 61           &             &            &              \\
            242          & 58           &             & 0, 0, 0, 2 &              \\
            243          & 58           &             & 0, 0, 4    &              \\
            244          & 52           &             &            &              \\
            245          & 46           &             &            &              \\
            249          & 108          & 34, 4, 2, 1 &            &              \\
            256          & 70           & 22, 3, 2, 1 &                           \\
            257          & 65           &             & 0, 0, 6    &              \\
            258          & 65           &             & 0, 1, 3    &              \\
            259          & 65           &             & 0, 5       &              \\ \bottomrule
        \end{tabular}
    \end{minipage}
\end{table}
\restoregeometry

\newpage
\newgeometry{left=1cm,bottom=0.1cm,right=.3cm}
\begin{table}
    \begin{minipage}{.4\textwidth}     \caption{Data for the remaining $3$-generated groups of order 64} \label{tbl:3Generated}
        \begin{tabular}{cccccc}
            \toprule
            $G$          & $\dim \HH^1$ & K           & Q       & $\dim \rH^2$ & G-L \\ \midrule
            60           & 108          &             &         &              &     \\
            61           & 96           & 28, 4, 1    &         &              &     \\
            62           & 96           & 28, 6, 1    &         &              &     \\
            63           & 84           & 28, 6, 1    &         & 6            &     \\
            64           & 84           & 28, 8, 1    &         &              &     \\
            \textbf{65}  & 84           & 28, 5, 1    &         &              & 8   \\
            66           & 96           & 28, 5, 1    &         &              &     \\
            67           & 98           &             &         &              &     \\
            68           & 84           & 28, 6, 1    &         & 5            &     \\
            69           & 86           &             &         &              &     \\
            \textbf{70}  & 84           & 28, 5, 1    &         &              & 8   \\
            71           & 88           & 28, 5, 1    &         &              &     \\
            72           & 84           & 28, 5, 1    &         &              & 16  \\
            73           & 72           &             &         &              &     \\
            74           & 68           & 22, 5, 1    &         & 7            &     \\
            75           & 70           &             &         &              &     \\
            76           & 66           & 22, 5, 1    &         & 6            &     \\
            77           & 68           & 22, 4, 1    &         &              &     \\
            78           & 68           & 22, 6, 1    &         &              &     \\
            79           & 66           & 22, 5, 1    &         & 5            &     \\
            80           & 68           & 22, 5, 1    &         & 6            &     \\
            81           & 66           & 22, 6, 1    &         &              &     \\
            82           & 66           & 22, 8, 1    &         &              &     \\
            88           & 88           & 28, 6, 2, 1 &         &              &     \\
            89           & 80           &             &         &              &     \\
            91           & 60           & 22, 4, 1    &         &              &     \\
            94           & 60           & 22, 4, 2, 1 &         &              &     \\
            98           & 64           & 22, 4, 2, 1 & 0, 0, 2 &              &     \\
            99           & 64           & 22, 5, 2, 1 &         &              &     \\
            100          & 60           & 22, 5, 2, 1 & 0, 0, 1 &              &     \\
            102          & 60           & 22, 5, 2, 1 & 0, 0, 2 &              &     \\
            \textbf{104} & 76           & 28, 6, 2, 1 &         &              &     \\
            \textbf{105} & 76           & 28, 6, 2, 1 &         &              &     \\
            109          & 58           &             &         &              &     \\
            111          & 54           &             &         &              &     \\
            113          & 76           & 28, 8, 2, 1 &         &              &     \\
            114          & 66           & 28, 8, 2, 1 &         &              &     \\
            116          & 78           &             &         &              &     \\
            117          & 68           & 28, 8, 2, 1 &         &              &     \\
            121          & 56           & 22, 6, 2, 1 &         &              &     \\
            122          & 52           &             & 0, 1    &              &     \\
            123          & 60           & 22, 6, 2, 1 &         &              &     \\
            125          & 52           &             & 0, 3    &              &     \\
            127          & 64           & 28, 8, 2, 1 &         &              &     \\
            128          & 56           & 19, 5, 2, 1 &         &              &     \\
            129          & 52           &             & 0, 0, 2 &              &     \\
            130          & 52           &             & 0, 0, 3 &              &     \\
            131          & 55           &             &         &              &     \\
            132          & 51           &             & 0, 0, 1 &              &     \\
            133          & 51           &             & 0, 0, 2 &              &     \\ \bottomrule
        \end{tabular}
    \end{minipage}
    \begin{minipage}{.05\textwidth}
        $\phantom{a}$
    \end{minipage}
    \begin{minipage}{.35\textwidth}  
        \vspace*{1.6cm}
        \begin{tabular}{ccccc}
            \toprule
            $G$          & $\dim \HH^1$ & K           & Q       & $\dim \rH^2$ \\ \midrule
            134          & 40           &             & 0, 0, 5 &              \\
            135          & 40           &             & 0, 2, 1 &              \\
            136          & 39           &             & 0, 0, 4 &              \\
            137          & 39           &             & 0, 3    &              \\
            138          & 47           &             &         &              \\
            139          & 45           & 16, 5, 1    &         &              \\
            140          & 50           & 19, 5, 2, 1 & 0, 0, 3 &              \\
            141          & 46           & 19, 5, 2, 1 &         &              \\
            \textbf{142} & 48           & 19, 5, 2, 1 &         & 5            \\
            143          & 44           &             &         & 5            \\
            144          & 49           & 19, 5, 2, 1 & 0, 0, 2 &              \\
            145          & 45           & 19, 5, 2, 1 &         &              \\
            146          & 63           &             &         &              \\
            147          & 64           & 22, 4, 2, 1 & 0, 0, 3 &              \\
            148          & 62           &             &         &              \\
            149          & 41           &             &         &              \\
            150          & 42           &             &         &              \\
            151          & 40           &             & 0, 0, 1 &              \\
            152          & 37           & 16, 4, 2, 1 &         &              \\
            153          & 38           & 16, 4, 2, 1 &         &              \\
            154          & 36           & 16, 4, 2, 1 &         &              \\
            \textbf{155} & 48           & 19, 5, 2, 1 &         & 5            \\
            \textbf{156} & 44           &             &         & 4            \\
            \textbf{157} & 48           & 19, 5, 2, 1 &         & 5            \\
            \textbf{158} & 44           &             &         & 4            \\
            159          & 48           & 19, 5, 2, 1 &         & 4            \\
            \textbf{160} & 44           &             &         & 4            \\
            161          & 50           & 19, 5, 2, 1 & 0, 0, 2 &              \\
            162          & 50           & 19, 6, 2, 1 &         &              \\
            163          & 46           & 19, 6, 2, 1 &         &              \\
            164          & 49           & 19, 6, 2, 1 &         &              \\
            165          & 49           & 19, 5, 2, 1 & 0, 0, 1 &              \\
            166          & 45           & 19, 6, 2, 1 &         &              \\
            \textbf{167} & 50           & 22, 6, 2, 1 &         &              \\
            \textbf{168} & 48           & 22, 6, 2, 1 &         & 4            \\
            169          & 49           & 22, 6, 2, 1 &         &              \\
            170          & 37           & 16, 6, 2, 1 &         & 4            \\
            171          & 38           & 16, 6, 2, 1 &         &              \\
            \textbf{172} & 36           & 16, 6, 2, 1 &         &              \\
            \textbf{173} & 50           & 22, 6, 2, 1 &         &              \\
            174          & 52           &             & 0, 0, 4 &              \\
            \textbf{175} & 48           & 22, 6, 2, 1 &         & 5            \\
            \textbf{176} & 50           & 22, 6, 2, 1 &         &              \\
            177          & 39           &             & 0,0,3   &              \\
            178          & 37           & 16, 6, 2, 1 &         & 5            \\
            \textbf{179} & 48           & 22, 6, 2, 1 &         & 4            \\
            \textbf{180} & 48           & 22, 6, 2, 1 &         & 4            \\
            \textbf{181} & 48           & 22, 6, 2, 1 &         & 5            \\
            \textbf{182} & 36           & 16, 6, 2, 1 &         &              \\
            190          & 34           &             &         &              \\
            191          & 32           &             &         &              \\ \bottomrule
        \end{tabular}
    \end{minipage}
\end{table}
\restoregeometry

\bibliographystyle{abbrv}
\bibliography{references}



\end{document}